\documentclass[10pt,twoside]{amsart}
\usepackage[pagewise]{lineno}
\usepackage{euscript,graphicx,epstopdf,amscd,amsgen,amsfonts,amssymb,latexsym,
amsmath,amsthm,graphicx,mathrsfs,times,color,overpic,tikz, geometry}
\makeatletter
\@namedef{subjclassname@2020}{%
  \textup{2020} Mathematics Subject Classification}
\makeatother

\newtheorem{theorem}{Theorem}[section]
\newtheorem{definition}[theorem]{Definition}
\newtheorem{prop}[theorem]{Proposition}

\newtheorem{lemma}[theorem]{Lemma}

\newtheorem*{theorem-non}{Theorem}

\theoremstyle{definition}

\def\HD{{\rm HD}}
\def\H{{\mathcal{H}}}

\def\R{\mathbb{R}}

\newcommand{\M}{\widetilde{M}}
\newcommand{\Bi}{\mathfrak{B}}
\newcommand{\bi}{\partial_{\infty}}

\DeclareMathSymbol{\varnothing}{\mathord}{AMSb}{"3F}
\makeatletter
\@namedef{subjclassname@2020}{%
  \textup{2020} Mathematics Subject Classification}
\makeatother

\begin{document}
\title{On linear escape and the dimension of limit sets in variable negative curvature}


\author[D. Pizarro]{Daniel Pizarro}
\address{IMA, Pontificia Universidad Cat\'olica de Valpara\'iso, Blanco Viel 596, Valpara\'iso, Chile.}
\email{daniel.pizarro@pucv.cl}

\author[F. Riquelme]{Felipe Riquelme}
\address{IMA, Pontificia Universidad Cat\'olica de Valpara\'iso, Blanco Viel 596, Valpara\'iso, Chile.}
\email{felipe.riquelme@pucv.cl}



\date{\today}


\begin{abstract}

In 2004, Bishop proved that for Kleinian groups acting on hyperbolic space, the Hausdorff dimension of the limit set is completely determined by two extremal dynamical behaviors: recurrent geodesics and geodesics escaping linearly to infinity. In this paper, we extend this phenomenon to arbitrary discrete groups of isometries of complete simply connected Riemannian manifolds with pinched negative sectional curvatures $-b^2\leq k\leq -1$. More precisely, we show that the Hausdorff dimension of the limit set coincides with the maximum of the Hausdorff dimensions of the radial limit set and the linear escape limit set. 
\end{abstract}

\maketitle

\section{Introduction}

The behavior of geodesic orbits in negatively curved manifolds has been widely studied for more than a century. To the best of our knowledge, Hadamard \cite{had:1898} was the first to exhibit non-compact negatively curved Riemannian manifolds admitting non-trivial bounded geodesic orbits. Since then, substantial effort has been devoted to understanding the size and structure of the set of bounded or recurrent geodesic orbits. For instance, Patterson \cite{Pat:1976} proved that, for convex-cocompact hyperbolic surfaces, the Hausdorff dimension of the set of bounded geodesic orbits coincides with the critical exponent of the fundamental group. This result was later extended in several directions: Stratmann \cite{st97} treated geometrically finite hyperbolic manifolds, Bishop and Jones \cite{bishop1997hausdorff} considered arbitrary hyperbolic three-manifolds, and Paulin \cite{paulin1997critical} established the corresponding result for complete Riemannian manifolds with pinched negative sectional curvatures. In particular, in this general setting, the Hausdorff dimension of the recurrent part of the limit set is given by the critical exponent.

A complementary problem is to understand the size of the non-recurrent part of the dynamics. This question becomes especially relevant for geometrically infinite manifolds, where escaping geodesics may contribute non-trivially to the Hausdorff dimension of the limit set. Matsuzaki \cite{Mat:2000} exhibited hyperbolic surfaces for which the Hausdorff dimension of the recurrent part is strictly smaller than the Hausdorff dimension of the whole limit set. Groups with this property are known as \emph{discrepancy groups}, or simply \emph{d-groups}. For such groups, the geometry of escaping geodesics is not a negligible phenomenon: it is necessary in order to recover the full dimension of the limit set. 

Motivated by Bishop's work on linear escape limit sets \cite{bishop2004linear}, we show that this missing dimension is completely accounted for by geodesic rays escaping at linear speed. More precisely, we prove that for any discrete group of isometries of a complete simply connected Riemannian manifold with pinched negative sectional curvatures, the Hausdorff dimension of the limit set is determined by two extremal dynamical regimes: recurrent geodesics and linearly escaping geodesics. Thus, intermediate escaping behavior does not contribute additional Hausdorff dimension.

Our work is primarily motivated by Bishop's study of linear escape limit sets \cite{bishop2004linear}. Related connections between escape rates and boundary approximation properties were also explored by Lundh \cite{lundh2003} through explicit hyperbolic geometric computations. In the variable curvature setting, however, a different geometric approach is required. In place of model-dependent hyperbolic formulas, the proof relies on CAT$(-1)$ comparison geometry together with shadow estimates in the boundary at infinity. The key geometric ingredient is Proposition \ref{prop:key}, which relates sublinear escape behavior to a weak form of recurrence.\\

Let $\M$ be a complete simply connected $n$-dimensional Riemannian manifold with pinched sectional curvatures $-b^2\leq k\leq -1$, let $\partial_\infty \M$ be its boundary at infinity, and let $\Gamma$ be a non-elementary discrete subgroup of isometries of $\M$. Given a basepoint $o\in\M$, the limit set $\Lambda_\Gamma\subset\partial_\infty\M$ is defined by
\[
\Lambda_\Gamma=\overline{\Gamma\cdot o}\setminus \Gamma\cdot o.
\]

For $\xi\in\partial_\infty\M$, let $\xi_t\in\M$ denote the point on the geodesic ray $[o,\xi)$ at distance $t$ from $o$, and define
\[
\Delta(\xi_t)=d(\xi_t,\Gamma\cdot o).
\]

The radial limit set $\Lambda_\Gamma^r$ is defined as
\[
\Lambda_\Gamma^r=
\left\{
\xi\in\Lambda_\Gamma:
\liminf_{t\to\infty}\Delta(\xi_t)<\infty
\right\}.
\]

Geometrically, points in $\Lambda_\Gamma^r$ correspond to geodesic rays in $\M/\Gamma$ that return infinitely often to a compact set. On the other hand, the transient limit set
\[
\Lambda_\Gamma^\tau=
\left\{
\xi\in\Lambda_\Gamma:
\lim_{t\to\infty}\Delta(\xi_t)=\infty
\right\}
\]
corresponds to geodesic rays eventually leaving every compact subset of $\M/\Gamma$. Among transient limit points, we distinguish those escaping at positive linear speed.

\begin{definition}
Let $\alpha>0$. The \emph{$\alpha$-linear escape limit set} is defined as
\[
\Lambda_\Gamma^l(\alpha)=
\left\{
\xi\in\Lambda_\Gamma:
\liminf_{t\to\infty}\frac{\Delta(\xi_t)}{t}>\alpha
\right\}.
\]
The \emph{linear escape limit set} is defined by
\[
\Lambda_\Gamma^l=
\bigcup_{0<\alpha<1}\Lambda_\Gamma^l(\alpha).
\]
\end{definition}

Our main result shows that, from the point of view of Hausdorff dimension, the limit set is completely determined by two extremal dynamical behaviors: recurrence and linear escape.

\begin{theorem}\label{main1}
Let $\M$ be a complete negatively curved $n$-dimensional Riemannian manifold with pinched sectional curvatures $-b^2\leq k\leq -1$. If $\Gamma$ is a discrete subgroup of isometries of $\M$, then
\[
\HD(\Lambda_\Gamma)
=
\max\left\{
\HD(\Lambda_\Gamma^r),
\HD(\Lambda_\Gamma^l)
\right\}.
\]
\end{theorem}

\subsection*{Acknowledgements}
 Both authors were supported by Agencia Nacional de Investigaci\'on y Desarrollo, grant FONDECYT Regular 1231257.

\section{Preliminaries}

\subsection{Geometry} Let $\M$ be a complete simply connected $n$-dimensional Riemannian manifold with pinched sectional curvatures $-b^2\leq k\leq -1$, with $b\geq 1$. The \emph{boundary at infinity} $\bi\M$ is the set of equivalent classes of asymptotic geodesic rays on $\M$. Endowing $\M\cup \bi\M$ with the cone topology makes $\M$ homeomorphic to the open unit $n$-dimensional Euclidean ball and $\bi\M$ homeomorphic to the unit $(n-1)$-dimensional sphere. We endow $\bi\M$ with the Gromov-Bourdon distance, which induces the cone topology, as follows. Let $z\in\M$. The \emph{Gromov product} at $z$ between $w_1,w_2\in\M$ is defined as
\[
\langle w_1,w_2\rangle_z =\frac{1}{2}\left(d_{\M}(w_1,z)+d_{\M}(z,w_2)-d_{\M}(w_1,w_2)\right).
\]
The Gromov product at $z$ is extended to the boundary as the limit
\[
\langle \xi_1,\xi_2\rangle_z=\lim_{w_1\to \xi_1, w_2\to \xi_2}\langle w_1,w_2\rangle_z.
\]
where $\xi_1,\xi_2\in\bi\M$. Finally, given $z\in\M$, the \emph{Gromov-Bourdon visual distance} $\rho_z$ on $\bi\M$ is defined as
\[
\rho_{z}(\xi_1,\xi_2)=\left\{\begin{matrix}
e^{-\langle \xi_1,\xi_2\rangle_{z}} & \textrm{if} & \xi_1\neq \xi_2,\\ 
0 & \textrm{if} & \xi_1 = \xi_2.
\end{matrix}\right.
\]
These distances are mutually conformal, and more precisely, Lipschitz-equivalent (see for instance \cite{Bou95}). 

From now on, we fix a point $o\in\M$, which will be called \emph{origin}, and $\rho:=\rho_o$ will be the visual distance seen from $o$. Balls of radius $r>0$ and center $\xi\in\partial_\infty\M$ will be denoted by $\mathfrak{B}(\xi,r)$. A ball of radius $R>0$ and center $z\in\widetilde{M}$ will be denoted as $B(z,R)$, as usual.

\begin{definition} Given $z\in\widetilde{M}$ and $R>0$, we define the \textit{shadow at infinity} $\mathcal{O}_{o}(z,R)$ of $B(z,R)$ as
\[
\mathcal{O}_{o}(z,R):=\left\{\xi\in\bi\M : [o,\xi)\cap B(z,R)\neq\emptyset\right\}.
\]
\end{definition}

Shadows and boundary balls are related by the following result proved by Kaimanovich in \cite{kaimanovich1990invariant}.

\begin{theorem}[Kaimanovich]\label{teo:kai} There exists a constant $c\geq 1$ such that for every $\xi\in\bi\M$ and $t>0$, we have
\[
\Bi(\xi,c^{-1}e^{-t}) \subset \mathcal{O}_{o}(\xi_t,1) \subset \Bi(\xi,ce^{-t}),
\]
where $\xi_t\in\M$ is the point on the geodesic ray $[o,\xi)$ at distance $t$ from $o$. 
\end{theorem}

The action by isometries of $\M$ extends naturally to the boundary as homeomorphisms of $\bi\M$. Accordingly, isometries are classified according to the number of fixed points in $\M$ or $\bi\M$. An isometry is called \emph{hyperbolic} if it fixes exactly two points in $\bi\M$, \emph{parabolic} if it fixes exactly one point in $\bi\M$, and \emph{elliptic} if it fixes exactly one point in $\M$. 

Recall that a \emph{Kleinian group} is a discrete group of isometries of $\M$. Since every elliptic element in a Kleinian group has finite order, torsion-free Kleinian groups consist exclusively of hyperbolic or parabolic elements.

\begin{definition} The \emph{Poincar\'e series} of a Kleinian group $\Gamma<\textrm{Isom}(\M)$ is defined as
\[
P_\Gamma(s):=\sum_{\gamma\in\Gamma} e^{-s d(o,\gamma\cdot o)}.
\]
\end{definition}
The Poincar\'e series converges for $s>\delta_\Gamma$ and diverges for $s<\delta_\Gamma$, where $\delta_\Gamma$ is the non-negative real number defined as
\[
\delta_\Gamma:=\limsup_{R\to\infty} \frac{1}{R}\log \#\{ \gamma\in\Gamma : d(o,\gamma\cdot o)\leq R  \}.
\] 

To end this subsection, we establish two geometric lemmas that will be useful later. 

\begin{lemma}\label{lem:main1} Let $z,w\in\widetilde{M}$ be such that $d(z,w)\leq \beta d(o,z)$, where $0<\beta<1$. Then
$$\frac{1}{1+\beta}d(o,w)\leq d(o,z)\leq\frac{1}{1-\beta}d(o,w).$$
\end{lemma}
\begin{proof}
Using triangle inequality, we get
\begin{eqnarray*}
    d(o,w) &\leq& d(o,z) + d(z,w)\\
    &\leq& d(o,z) + \beta d(o,z)\\
    &\leq& (1+\beta) d(o,z),
\end{eqnarray*}
so the first inequality follows. The second inequality follows from a similar argument.
\end{proof}
 
For the second lemma, given $z\in\widetilde{M}\setminus\{o\}$, we will denote by $z^+\in\partial_\infty\widetilde{M}$ to the extremity point of the geodesic ray starting at $o$ which contains $z$.   
 
\begin{lemma}\label{lem:main2} Let $z,w\in\widetilde{M}$ be such that $d(z,w)\leq \beta d(o,z)$, where $0<\beta<1/5$. Then, there is $p\in [o,w^+)$ such that $d\big(p,[o,z^+)\big)\leq 1$ and 
\[
d(o,p) \geq \frac{1-2\beta-\beta^2}{1-\beta^2} d(o,w).
\]
\end{lemma}
\begin{proof}
By Lemma \ref{lem:main1}, we know that $d(o,z)\geq\frac{1}{1+\beta}d(o,w)=:a$. Let $z'\in[o,z]$ and $w'\in[o,w]$ be such that $d(o,z')=a=d(o,w')$. Set $b:=d(z',w')$. Since $\M$ is a $CAT(-1)$-space, we can consider a comparison triangle $\triangle\bar{x}\bar{y}\bar{z}$ in $\mathbb{H}$, where $d_1(\bar{x},\bar{y})=a=d_1(\bar{x},\bar{z})$, and $d_1(\bar{y},\bar{z})=b$ (see Figure \ref{fig:fig4} below).

\noindent
We claim that, if $\beta<1/5$, then the angle $\theta$ at $\bar{x}$ of $\triangle\bar{x}\bar{y}\bar{z}$ is in $(0,\pi/2]$. In order to prove it, it is enough to show that $b\leq a$. Indeed, by triangle inequality, we have
\[
b = d(z',w')\leq d(z',z)+d(z,w)+d(w,w').
\]

\begin{figure}[h]
    \centering
    \begin{overpic}[width=1.03\linewidth]{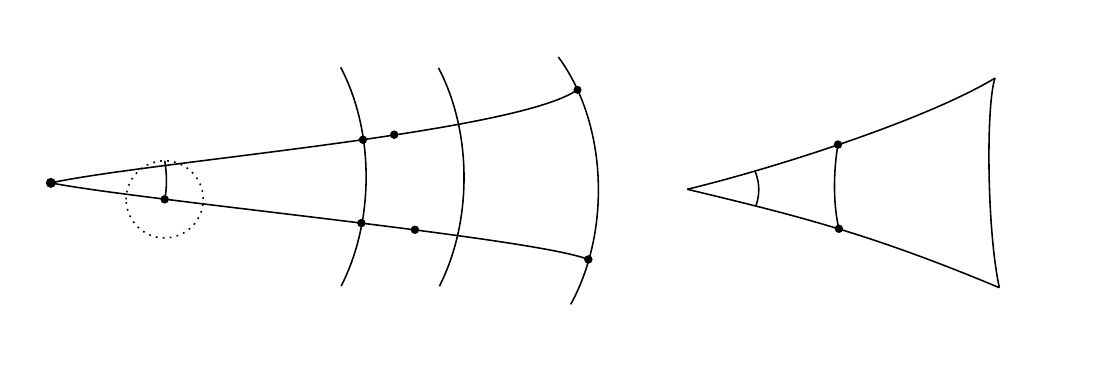}
    \put(5,27){$\widetilde{M}$}
    \put(3,15){$o$}
    \put(14,13.3){$p$}
    \put(9,14){$c$}
    \put(15.5,16){$1$}
    \put(19,22){$a$}
    \put(19,10.5){$a$}
    \put(34,16){$b$}
    \put(33.4,21.7){$z'$}
    \put(36,22.2){$z$}
    \put(33,11){$w'$}
    \put(37.5,10.6){$w$}
    \put(54,8.3){$w^{+}$}
    \put(53.3,26){$z^{+}$}
    \put(52,3){$\partial_{\infty} \widetilde{M}$}
    \put(63,27){$\mathbb{H}$}
    \put(60.5,15.5){$\overline{x}$}
    \put(75.5,10){$\overline{p}$}
    \put(75.5,22){$\overline{q}$}
    \put(91,27){$\overline{z}$}
    \put(91.5,5.2){$\overline{y}$}
    \put(91,16){$b$}
    \put(76.5,15.7){$1$}
    \put(77,26){$a$}
    \put(77,6){$a$}
    \footnotesize
    \put(66.5,15.7){$\theta$}
    \normalsize
    \put(69,12){$c$}
    \end{overpic}
    \caption{Comparison theorem}
    \label{fig:fig4}
\end{figure}
\noindent
Moreover, by Lemma \ref{lem:main1}, we know that
\[
d(z',z)\leq \left(\frac{1}{1-\beta}-\frac{1}{1+\beta}\right)d(o,w),
\quad
d(z,w)\leq \frac{\beta}{1-\beta}d(o,w)
\]
and
\[
d(w,w')\leq \left(1-\frac{1}{1+\beta}\right)d(o,w).
\] 
Hence $d(z',w')\leq \frac{4\beta}{(1-\beta)(1+\beta)}d(o,w)$, or equivalently, $d(z',w')\leq \frac{4\beta}{1-\beta}d(o,w')$. In particular, if $\beta<1/5$, we get $b\leq a$.

Since $0<\theta<\pi/2$, there exists $\bar{p}\in [\bar{x},\bar{y}]$ at distance $1$ from the geodesic segment $[\bar{x},\bar{z}]$. Set $\bar{q}\in [\bar{x},\bar{z}]$ such that $d_1(\bar{p},\bar{q})=1$. Let $c:= d_1(\bar{x},\bar{p})$. By the hyperbolic law of sines, we have

\[
\frac{1}{\sinh(c)}=\frac{\sin(\theta)}{\sinh(1)}
\quad
\mbox{and}
\quad
\frac{1}{\sinh(a)}= \frac{\sin(\theta/2)}{\sinh(b/2)}.
\]
Thus
\[
\frac{\sinh(1)}{\sinh(c)}=\sin(\theta)=2\cos(\theta/2)\frac{\sinh(b/2)}{\sinh(a)}.
\]
Since $\cos(\theta/2)\leq 1$ and the hyperbolic sinus verifies $e^t-1\leq 2\sinh(t) \leq e^t$, one can easily check that
\[
2C(e^a-1)\leq e^ce^{b/2},
\]
where $C=\sinh(1)/2$. Moreover, if $a\geq\ln(2)$, one has $2(e^a-1)\geq e^a$, so $Ce^a\leq e^{c+b/2}$. In particular, if $b/2\geq -\ln(C)$, we get
\[
c\geq a-b.
\] 
Since $a=\frac{1}{1+\beta}d(o,w)$ and $b=d(z',w')\leq d(z,w)\leq \frac{\beta}{1-\beta}d(o,w)$, we conclude
\[
c \geq \frac{1-2\beta-\beta^2}{1-\beta^2}d(o,w).
\]

Let $p\in[o,w]$ be the unique point such that $d(o,p)=c$ and let $q\in [o,z]$ be the unique point such that $d(o,q)=d_1(\bar{x},\bar{q})$. Then, by comparison, we have $1=d_1(\bar{p},\bar{q})\geq d(p,q)$. In particular, the distance between $p$ and the geodesic segment $[o,z]$ is less than 1. This concludes the proof of the lemma. 
\end{proof}

\subsection{Hausdorff dimension} 

Let $E$ be a Borel subset of $\bi\M$. For every $s>0$ and $\delta>0$, set
\[
\H^s_\delta(E)=\inf \sum_i r_i^s,
\]
where the infimum is taken over all finite covering $\{\Bi(\xi_i,r_i)\}$ of $E$ with $r_i\leq\delta$. Also define the (outer) \emph{Hausdorff measure} of dimension $s$ as
\[
\H^s(E)=\lim_{\delta\to 0} \H^s_\delta(E).
\]
It is well known that the map $\H^s(E)$ is infinite until it vanishes after some $s^\star\in\R$. 

\begin{definition} The Hausdorff dimension of a Borel set $E\subset\bi\M$ is the value
\[
\emph{HD}(E)=\sup\{s:\H^s(E)=\infty\} = \inf\{s:\H^s(E)=0\}.
\]
\end{definition}

We stress the fact that the Hausdorff dimension does not depends on the Gromov-Bourdon distance since all of them are conformal. Using this notation, Bishop-Jones theorem can be stated as follows (see \cite{bishop1997hausdorff} and \cite{paulin1997critical})

\begin{theorem}[Bishop-Jones] Let $\Gamma$ be a Kleinian group of isometries of a complete negatively curved $n$-dimensional Riemannian manifold $\widetilde{M}$ with pinched sectional curvatures $-b^2\leq k\leq -1$. Then
\[
\HD(\Lambda^r_\Gamma)=\delta_\Gamma.
\]
\end{theorem}  

\section{Dimension of the limit set}

In this section $\Gamma<\textrm{Isom}^+(\M)$ will always be a Kleinian group, $\Lambda$ will denote its limit set, $\Lambda^r$ its radial limit set and $\Lambda^l$ its linear escape limit set. To study limit points which do not escape linearly, we introduce a weakened form of recurrence. Recall that radial limit points are characterized by infinitely many approximations by orbit shadows at the natural exponential scale $e^{-d(o,\gamma o)}$. In the absence of linear escape, one still obtains infinitely many such approximations, although at a slower exponential rate. This motivates the following definition.

\begin{definition}\label{def:wrls} Let $\kappa>0$. The \emph{$\kappa$-weakly recurrent limit set} $\Lambda^r(\kappa)$ is defined as 
\[
\Lambda^r(\kappa)= \bigcup_{c>0} \left\{ \xi\in\Lambda : \xi\in\Bi((\gamma \cdot o)^+,ce^{-\frac{1}{1+\kappa}d(o,\gamma\cdot o)}) \textrm{ for infinitely many }\gamma\in\Gamma \right\}.
\]
\end{definition}

One can verify that, if $0<\kappa_1\leq\kappa_2$, then $\Lambda^r(\kappa_2)\subset \Lambda^r(\kappa_1)$. Moreover, for every $\kappa>0$, one has 
\begin{equation}\label{eq:weaklyrecurrent1}
\Lambda^r\subset \Lambda^r(\kappa).
\end{equation}
The following proposition relates weakly recurrent limit sets with linear escape limits sets. Let $\eta(\beta)=\frac{1-2\beta-\beta^2}{1-\beta^2}$

\begin{prop}\label{prop:key} Let $0<\alpha<\beta<1/5$. For every $\xi\in\Lambda\setminus\Lambda^l(\alpha)$ there exist $c\geq 1$ and a sequence $(\gamma_n)_n\subset\Gamma$, such that
\[
\xi\in\Bi\left( (\gamma_n\cdot o)^+, ce^{- \eta(\beta)d(o,\gamma_n\cdot o)} \right)
\]
for every $n\geq 1$. In particular, we have $\Lambda\setminus\Lambda^l(\alpha)\subset \Lambda^r\left(\frac{1-\eta(\beta)}{\eta(\beta)}\right)$.
\end{prop}
\begin{proof}
Let $\xi\in\Lambda\setminus\Lambda^l(\alpha)$. Then, by definition, we have
\[
\liminf_{t\to\infty} \frac{1}{t}\Delta(\xi_t)\leq \alpha.
\]
Since $\beta>\alpha$, there are sequences $(t_n)_n\nearrow\infty$ and $(\gamma_n)_n\subset\Gamma$ with
\[
d(\xi_{t_n},\gamma_n\cdot o) \leq \beta t_n,
\]    
for all $n\geq 1$. By Theorem \ref{teo:kai}, $\xi\in\mathcal{O}_o(\xi_{t_n},1)\subset\Bi(\xi,ce^{-t_n})$ for some universal constant $c\geq 1$. Note that $t_n=d(o,\xi_{t_n})$, so we can apply Lemma \ref{lem:main1} for $z=\xi_{t_n}$ and $w=\gamma_n\cdot o$, to get
\[
d(\xi_{t_n},\gamma_n\cdot o) \leq \beta d(o,\xi_{t_n}).
\]
By Lemma \ref{lem:main2}, there is $p_n\in [o,(\gamma_n\cdot o)^+)$ such that 
\[
d(o,p_n)\geq \eta(\beta)d(o,\gamma_n \cdot o) \quad \mbox{and} \quad d\big(p_n,[o,\xi)\big)\leq 1.
\]
In particular, the last inequality ensures that $\xi\in \mathcal{O}_o(p_n,1)$. Finally, since 
\begin{eqnarray*}
\mathcal{O}_o(p_n,1) &\subset& \Bi((\gamma_n\cdot o)^+,ce^{-d(o,p_n)}),\\
&\subset& \Bi\left( (\gamma_n\cdot o)^+, ce^{-\eta(\beta) d(o,\gamma_n\cdot o)} \right),
\end{eqnarray*}
the conclusion of this proposition follows.
\end{proof}

Proposition \ref{prop:key} shows that limit points escaping slower than linearly belong to weakly recurrent limit sets associated to weaker exponential approximation scales. Therefore, to prove Theorem \ref{main1}, it suffices to estimate the Hausdorff dimension of weakly recurrent limit sets.

Observe that if $\Gamma$ is not a $d$-group, then
\[
\HD(\Lambda)=\HD(\Lambda^r),
\]
and the theorem is immediate. We therefore restrict to the case where $\Gamma$ is a $d$-group.

The following proposition was proved by Falk and Stratmann \cite{falk2004remarks} in the hyperbolic setting. The same argument applies in variable negative curvature, but for completeness we include the proof. Set
\[
\delta^\star=
\frac{\HD(\Lambda)-\delta_\Gamma}{\delta_\Gamma}.
\]

\begin{prop}[Falk-Stratmann]\label{prop:FS} Assume $\Gamma$ is a $d$-group and let $0<\kappa<\delta^\star$. Then
\[
\delta_\Gamma\leq\emph{HD}(\Lambda^r(\kappa))<\emph{HD}(\Lambda).
\]
\end{prop}
\begin{proof}
The first inequality follows from \eqref{eq:weaklyrecurrent1} and $\textrm{HD}(\Lambda^r)=\delta_\Gamma$. For the second one, let $0<\kappa<\delta^\star$. For any integer $m\geq 1$ and $\gamma\in\Gamma$, set $B_m(\gamma)= \Bi((\gamma \cdot o)^+,me^{-\frac{1}{1+\kappa}d(o,\gamma\cdot o)})$. Then, by definition
\[
\Lambda^r(\kappa)= \bigcup_{m\geq 1} \limsup\{B_m(\gamma), \gamma\in\Gamma \}.
\]

For each $m\geq 1$ and $R>0$, the family $\{B_m(\gamma):\gamma\in\Gamma, \ d(o,\gamma\cdot o)\geq R\}$ forms a covering of $\limsup\{B_m(\gamma), \gamma\in\Gamma \}$. Moreover, for any $\delta>0$, there is $R_0=R_0(m,\kappa)>0$ such that $\textrm{diam}(B_m(\gamma))<\delta$ if $d(o,\gamma\cdot o)\geq R_0$. Since
\[
\sum_{\gamma\in\Gamma, d(o,\gamma\cdot o)\geq R} (me^{-\frac{1}{1+\kappa}d(o,\gamma\cdot o)})^s<\infty \quad \textrm{for all} \quad s>(1+\kappa)\delta_\Gamma,
\]
we get $\mathcal{H}^s(\limsup\{B_m(\gamma), \gamma\in\Gamma \})=0$, for $s>(1+\kappa)\delta_\Gamma$. In particular, for each $m\geq 1$, we have
\[
\textrm{HD}(\limsup\{ B_m(\gamma), \gamma\in\Gamma \})\leq (1+\kappa)\delta_\Gamma.
\]
Recall that the Hausdorff dimension is stable under countable union of sets, so 
\[
\textrm{HD}\big(\Lambda^r(\kappa)\big)\leq (1+\kappa)\delta_\Gamma.
\]
Finally, since $\kappa<\delta^\star$, one can easily verify that $\textrm{HD}\big(\Lambda^r(\kappa)\big)<\textrm{HD}(\Lambda)$.
\end{proof}

\begin{proof}[Proof of Theorem \ref{main1}]

By Proposition \ref{prop:key}, we have
\[
\Lambda\setminus\Lambda^l(\alpha)
\subset
\Lambda^r\left(\frac{1-\eta(\beta)}{\eta(\beta)}\right),
\]
and therefore
\[
\HD\big(\Lambda\setminus\Lambda^l(\alpha)\big)
\leq
\HD\left(
\Lambda^r\left(\frac{1-\eta(\beta)}{\eta(\beta)}\right)
\right).
\]

Since $(1-\eta(\beta))/\eta(\beta)\to 0$ as $\beta\to 0$, there exists $\beta>0$ sufficiently small such that
\[
\frac{1-\eta(\beta)}{\eta(\beta)}<\delta^\star.
\]
Hence, Proposition \ref{prop:FS} implies that
\[
\HD\big(\Lambda\setminus\Lambda^l(\alpha)\big)<\HD(\Lambda).
\]
On the other hand, we have
\[
\Lambda=\big(\Lambda\setminus\Lambda^l(\alpha)\big)\cup\Lambda^l(\alpha),
\]
so it follows that
\[
\HD\big(\Lambda^l(\alpha)\big)=\HD(\Lambda)
\]
for every sufficiently small $\alpha>0$. In particular,
\[
\HD(\Lambda)=\max\{\HD(\Lambda^r),\HD(\Lambda^l)\}.
\]

\end{proof}

\subsection{Final thoughts}

Theorem \ref{main1} shows that, for discrepancy groups, the full Hausdorff dimension of the limit set is carried by linearly escaping limit points. This phenomenon contrasts sharply with the behavior of several canonical measures on the boundary at infinity.

Indeed, if $\Gamma$ is of divergence type, then the Patterson--Sullivan measure gives full measure to the radial limit set by the Hopf-Tsuji-Sullivan theorem. Since $\Lambda^l$ is disjoint from the radial limit set, it follows that the Patterson--Sullivan measure of $\Lambda^l$ vanishes.

Similarly, let $\mu$ be a non-elementary probability measure on $\Gamma$ with finite first moment, and let $\nu_\mu$ denote the associated harmonic measure on $\partial_\infty \widetilde M$. By sublinear tracking results for random walks on negatively curved spaces (see for instance \cite{Kai2000,MT2018}), $\nu_\mu$-almost every sample path stays at sublinear distance from a geodesic ray. More precisely, for $\nu_\mu$-almost every boundary point $\xi$, there exists a sample path $(w_n\cdot o)_n$ converging to $\xi$ and a sequence $t_n\to\infty$ such that
\[
d(w_n \cdot o,\xi_{t_n})=o(n),
\]
where $\xi$ denotes the boundary point determined by the random walk and $\xi_{t_n}\in[o,\xi)$ is the point at distance $t_n$ from $o$. Thus, we obtain
\[
\Delta(\xi_{t_n})=d(\xi_{t_n},\Gamma\cdot o)\leq d(\xi_{t_n},w_n \cdot o)=o(n).
\]
Moreover, by positive drift,
\[
\frac{1}{n}d(o,w_n \cdot o)\longrightarrow L>0,
\]
so that $t_n$ grows linearly in $n$. Therefore,
\[
\liminf_{t\to\infty}\frac{\Delta(\xi_t)}{t} \leq\liminf_{n\to\infty}\frac{\Delta(\xi_{t_n})}{t_n}=0.
\]
It follows that $\xi\notin\Lambda^l(\alpha)$ for every $\alpha>0$. Consequently, $\nu_\mu(\Lambda^l)=0.$

Thus, for discrepancy groups, the linear escape limit set may have full Hausdorff dimension while remaining negligible for both Patterson-Sullivan measures and harmonic measures associated to non-elementary random walks with finite first moment. The vanishing of harmonic measure on $\Lambda^l_\Gamma$ relies crucially on finite first moment assumptions through sublinear tracking results for random walks on negatively curved spaces. It would be interesting to understand whether harmonic measures associated to random walks with infinite first moment may charge the linear escape limit set.

\bibliographystyle{plain} 
\bibliography{bib.bib}

\end{document}